\documentclass{amsart}

\usepackage{amsthm}
\usepackage{a4wide}

\newtheorem{theorem}{Theorem}[section]

\numberwithin{equation}{section}

\begin{document}

\title[Coarse cohomology types of pure metacyclic fields]{Coarse cohomology types of pure metacyclic fields}

\address{Daniel C. Mayer\\Naglergasse 53\\8010 Graz\\Austria}
\email{algebraic.number.theory@algebra.at}
\urladdr{http://www.algebra.at}


\date{December 24, 2018}

\maketitle


\section{Introduction and Foundations}
\label{s:Intro}
\noindent
Let \(p\) be an odd prime number,
\(D\ge 2\) a \(p\)-th power free integer,
and \(\zeta_p\) be a primitive \(p\)-th root of unity.
The \textit{coarse cohomology types} of pure metacyclic fields \(N=\mathbb{Q}(\zeta_p,\sqrt[p]{D})\),
which are cyclic Kummer extensions of the cyclotomic field \(K=\mathbb{Q}(\zeta_p)\)
with relative automorphism group \(G=\mathrm{Gal}(N/K)=\langle\sigma\rangle\),
are based on the Galois cohomology of the unit group \(U_N\) viewed as a \(G\)-module.
The \textit{primary invariant} is the group \(H^0(G,U_N)\simeq U_K/N_{N/K}(U_N)\)
of order \(p^U\) with \(0\le U\le\frac{p-1}{2}\)
which is related to the group \(H^1(G,U_N)\simeq (U_N\cap\ker(N_{N/K}))/U_N^{1-\sigma}\)
of order \(p^{P}\) by the Takagi/Hasse/Iwasawa Theorem on the Herbrand quotient of \(U_N\):
\begin{equation}
\label{eqn:HerbrandQuot}
\#H^1(G,U_N)=\#H^0(G,U_N)\cdot\lbrack N:K\rbrack, \text{ respectively } P=U+1. 
\end{equation}
The \textit{secondary invariant} is a natural decomposition
of the group \(H^1(G,U_N)\simeq\mathcal{P}_{N/K}/\mathcal{P}_K\)
of \textit{primitive ambiguous principal ideals} of \(N/K\),
which can be viewed as principal ideals dividing the relative different \(\mathfrak{D}_{N/K}\)
and are therefore called \textit{differential principal factors} (DPF) of \(N/K\):
\begin{equation}
\label{eqn:NaturalDecomp}
\mathcal{P}_{N/K}/\mathcal{P}_K\simeq\mathcal{P}_{L/\mathbb{Q}}/\mathcal{P}_{\mathbb{Q}}
\times\left((\mathcal{P}_{N/K}/\mathcal{P}_K)\cap\ker(N_{N/L})\right),
\text{ respectively } U+1=A+R,
\end{equation}
where \(L=\mathbb{Q}(\sqrt[p]{D})\) denotes the real non-Galois pure subfield of degree \(p\) of \(N\).
The subgroup of \textit{absolute} DPF, \(\mathcal{P}_{L/\mathbb{Q}}/\mathcal{P}_{\mathbb{Q}}\), is of order \(p^A\),
and the subgroup of \textit{relative} DPF (the norm kernel), \((\mathcal{P}_{N/K}/\mathcal{P}_K)\cap\ker(N_{N/L})\), is of order \(p^R\)
\cite{Ma2a,Ma2c}.
We present \(p=7\) in comparison to \(p\in\lbrace 3,5\rbrace\).


\section{Pure Septic Fields}
\label{s:Septic}
\noindent
For \textit{pure septic} fields \(L=\mathbb{Q}(\sqrt[7]{D})\)
and their Galois closure \(N=\mathbb{Q}(\zeta_7,\sqrt[7]{D})\),
that is the case \(p=7\),
the \textit{coarse} classification of \(N\) according to the invariants \(U\) and \(A\) alone is
illustrated in Fig.
\ref{fig:Septic}:
The coarse types are
\(\alpha\), \(\beta\), \(\gamma\), \(\delta\) with \(U=3\),
\(\varepsilon\), \(\zeta\), \(\eta\) with \(U=2\),
\(\vartheta\), \(\iota\) with \(U=1\), and
\(\kappa\) with \(U=0\).
The possibility that the primitive seventh root of unity \(\zeta_7\)
occurs as the relative norm \(N_{N/K}(Z)\) of a unit \(Z\in U_N\)
will cause a splitting of all types with \(1\le U\le 2\),
similar to the splitting into \(\delta\)/\(\zeta\) and \(\varepsilon\)/\(\eta\)
in the pure quintic case of \S\
\ref{s:Quintic}.
Due to the existence of radicals in the pure septic field,
the \(\mathbb{F}_7\)-dimension \(A\) of the vector space of \textit{absolute} DPF
is at least one: \(1\le A\le 4\).


\begin{figure}[ht]
\caption{Classification of pure septic fields}
\label{fig:Septic}

{\tiny

\setlength{\unitlength}{1.0cm}
\begin{picture}(15,6)(-11,-10)




\put(-2,-4.8){\makebox(0,0)[cb]{\(U\)}}
\put(-2,-6){\vector(0,1){1}}
\put(-2,-9){\line(0,1){3}}
\multiput(-2.1,-9)(0,1){4}{\line(1,0){0.2}}

\put(-2.2,-6){\makebox(0,0)[rc]{\(3\)}}
\put(-2.2,-7){\makebox(0,0)[rc]{\(2\)}}
\put(-2.2,-8){\makebox(0,0)[rc]{\(1\)}}
\put(-2.2,-9){\makebox(0,0)[rc]{\(0\)}}

\put(3.2,-10){\makebox(0,0)[lc]{\(A\)}}
\put(2,-10){\vector(1,0){1}}
\put(-1,-10){\line(1,0){3}}
\multiput(-1,-10.1)(1,0){4}{\line(0,1){0.2}}

\put(-1,-10.2){\makebox(0,0)[ct]{\(1\)}}
\put(0,-10.2){\makebox(0,0)[ct]{\(2\)}}
\put(1,-10.2){\makebox(0,0)[ct]{\(3\)}}
\put(2,-10.2){\makebox(0,0)[ct]{\(4\)}}


\put(-1,-6){\circle*{0.2}}
\put(-1,-5.8){\makebox(0,0)[cb]{\(\alpha\)}}
\put(0,-6){\circle*{0.2}}
\put(0,-5.8){\makebox(0,0)[cb]{\(\beta\)}}
\put(1,-6){\circle*{0.2}}
\put(1,-5.8){\makebox(0,0)[cb]{\(\gamma\)}}
\put(2,-6){\circle{0.2}}
\put(2,-5.8){\makebox(0,0)[cb]{\(\delta\)}}

\put(-1,-7){\circle*{0.2}}
\put(-1,-6.8){\makebox(0,0)[cb]{\(\varepsilon\)}}
\put(0,-7){\circle*{0.2}}
\put(0,-6.8){\makebox(0,0)[cb]{\(\zeta\)}}
\put(1,-7){\circle{0.2}}
\put(1,-6.8){\makebox(0,0)[cb]{\(\eta\)}}

\put(-1,-7.9){\circle*{0.2}}
\put(-1,-7.7){\makebox(0,0)[cb]{\(\vartheta\)}}
\put(0,-7.9){\circle{0.2}}
\put(0,-7.7){\makebox(0,0)[cb]{\(\iota\)}}

\put(-1,-9){\circle{0.2}}
\put(-1,-9.2){\makebox(0,0)[ct]{\(\kappa\)}}


\end{picture}
}
\end{figure}
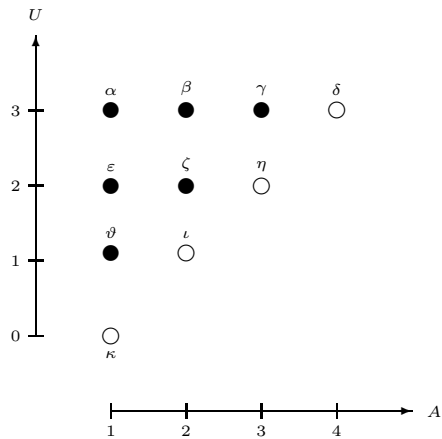


\section{Pure Quintic Fields}
\label{s:Quintic}
\noindent
For \textit{pure quintic} fields \(L=\mathbb{Q}(\sqrt[5]{D})\)
and their Galois closure \(N=\mathbb{Q}(\zeta_5,\sqrt[5]{D})\),
that is the case \(p=5\),
the \textit{coarse} classification of \(N\) according to the invariants \(U\) and \(A\) alone is closely related
to the classification of \textit{totally real dihedral} fields by Nicole Moser
\cite[Thm. III.5, p. 62]{Mo},
as illustrated in Figure
\ref{fig:MoserExtendedQuintic}:
The coarse types \(\alpha\), \(\beta\), \(\gamma\), \(\delta\), \(\varepsilon\)
are completely analogous in both cases.
Additional types \(\zeta\), \(\eta\), \(\vartheta\) are required for pure quintic fields,
because there arises the possibility that the primitive fifth root of unity \(\zeta_5\)
occurs as the relative norm \(N_{N/K}(Z)\) of a unit \(Z\in U_N\).
Due to the existence of radicals in the pure quintic case,
the \(\mathbb{F}_p\)-dimension \(A\) of the vector space of \textit{absolute} DPF
exceeds the corresponding dimension for totally real dihedral fields by one
\cite{Ma0,Ma2b,Ma2}.


\begin{figure}[ht]
\caption{Classification of totally real dihedral and pure quintic fields}
\label{fig:MoserExtendedQuintic}

{\tiny

\setlength{\unitlength}{1.0cm}
\begin{picture}(15,5)(-11,-10)




\put(-9,-5.8){\makebox(0,0)[cb]{\(U\)}}
\put(-9,-7){\vector(0,1){1}}
\put(-9,-8){\line(0,1){1}}
\multiput(-9.1,-8)(0,1){2}{\line(1,0){0.2}}

\put(-9.2,-7){\makebox(0,0)[rc]{\(1\)}}
\put(-9.2,-8){\makebox(0,0)[rc]{\(0\)}}

\put(-4.8,-10){\makebox(0,0)[lc]{\(A\)}}
\put(-6,-10){\vector(1,0){1}}
\put(-8,-10){\line(1,0){2}}
\multiput(-8,-10.1)(1,0){3}{\line(0,1){0.2}}

\put(-8,-10.2){\makebox(0,0)[ct]{\(0\)}}
\put(-7,-10.2){\makebox(0,0)[ct]{\(1\)}}
\put(-6,-10.2){\makebox(0,0)[ct]{\(2\)}}


\put(-8,-7){\circle*{0.2}}
\put(-8,-6.8){\makebox(0,0)[cb]{\(\alpha\)}}
\put(-7,-7){\circle*{0.2}}
\put(-7,-6.8){\makebox(0,0)[cb]{\(\beta\)}}
\put(-6,-7){\circle{0.2}}
\put(-6,-6.8){\makebox(0,0)[cb]{\(\gamma\)}}

\put(-8,-8){\circle*{0.2}}
\put(-8,-7.8){\makebox(0,0)[cb]{\(\delta\)}}
\put(-7,-8){\circle{0.2}}
\put(-7,-7.8){\makebox(0,0)[cb]{\(\varepsilon\)}}




\put(-2,-5.8){\makebox(0,0)[cb]{\(U\)}}
\put(-2,-7){\vector(0,1){1}}
\put(-2,-9){\line(0,1){2}}
\multiput(-2.1,-9)(0,1){3}{\line(1,0){0.2}}

\put(-2.2,-7){\makebox(0,0)[rc]{\(2\)}}
\put(-2.2,-8){\makebox(0,0)[rc]{\(1\)}}
\put(-2.2,-9){\makebox(0,0)[rc]{\(0\)}}

\put(2.2,-10){\makebox(0,0)[lc]{\(A\)}}
\put(1,-10){\vector(1,0){1}}
\put(-1,-10){\line(1,0){2}}
\multiput(-1,-10.1)(1,0){3}{\line(0,1){0.2}}

\put(-1,-10.2){\makebox(0,0)[ct]{\(1\)}}
\put(0,-10.2){\makebox(0,0)[ct]{\(2\)}}
\put(1,-10.2){\makebox(0,0)[ct]{\(3\)}}


\put(-1,-7){\circle*{0.2}}
\put(-1,-6.8){\makebox(0,0)[cb]{\(\alpha\)}}
\put(0,-7){\circle*{0.2}}
\put(0,-6.8){\makebox(0,0)[cb]{\(\beta\)}}
\put(1,-7){\circle{0.2}}
\put(1,-6.8){\makebox(0,0)[cb]{\(\gamma\)}}

\put(-1,-7.9){\circle*{0.2}}
\put(-1,-7.7){\makebox(0,0)[cb]{\(\delta\)}}
\put(0,-7.9){\circle{0.2}}
\put(0,-7.7){\makebox(0,0)[cb]{\(\varepsilon\)}}
\put(-1,-8.1){\circle*{0.2}}
\put(-1,-8.3){\makebox(0,0)[ct]{\(\zeta\)}}
\put(0,-8.1){\circle{0.2}}
\put(0,-8.3){\makebox(0,0)[ct]{\(\eta\)}}

\put(-1,-9){\circle{0.2}}
\put(-1,-9.2){\makebox(0,0)[ct]{\(\vartheta\)}}


\end{picture}
}
\end{figure}
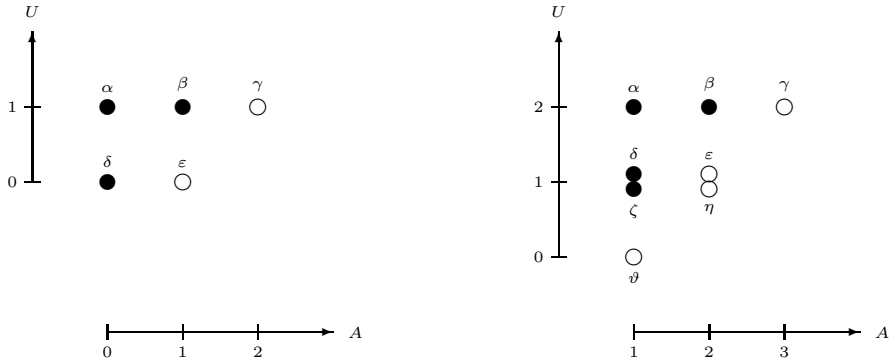


\section{Pure Cubic Fields}
\label{s:Cubic}
\noindent
For \textit{pure cubic} fields \(L=\mathbb{Q}(\sqrt[3]{D})\)
and their Galois closure \(N=\mathbb{Q}(\zeta_3,\sqrt[3]{D})\),
that is the case \(p=3\),
the \textit{coarse} classification of \(N\) according to the invariants \(U\) and \(A\) alone is closely related
to the classification of \textit{simply real dihedral} fields by Nicole Moser
\cite[Dfn. III.1 and Prop. III.3, p. 61]{Mo},
as illustrated in Figure
\ref{fig:MoserExtendedCubic}:
The coarse types \(\alpha\) and \(\beta\)
are completely analogous in both cases.
The additional type \(\gamma\) is required for pure cubic fields,
because there arises the possibility that the primitive cube root of unity \(\zeta_3\)
occurs as the relative norm \(N_{N/K}(Z)\) of a unit \(Z\in U_N\).
Due to the existence of radicals in the pure cubic case,
the \(\mathbb{F}_p\)-dimension \(A\) of the vector space of \textit{absolute} DPF
exceeds the corresponding dimension for simply real dihedral fields by one
\cite{AMITA,AMI,Ma}.


\begin{figure}[ht]
\caption{Classification of simply real dihedral and pure cubic fields}
\label{fig:MoserExtendedCubic}

{\tiny

\setlength{\unitlength}{1.0cm}
\begin{picture}(15,4)(-11,-9)




\put(-9,-5.8){\makebox(0,0)[cb]{\(U\)}}
\put(-9,-7){\vector(0,1){1}}
\put(-9,-7){\line(0,1){0}}
\multiput(-9.1,-7)(0,1){1}{\line(1,0){0.2}}

\put(-9.2,-7){\makebox(0,0)[rc]{\(0\)}}

\put(-5.8,-9){\makebox(0,0)[lc]{\(A\)}}
\put(-7,-9){\vector(1,0){1}}
\put(-8,-9){\line(1,0){1}}
\multiput(-8,-9.1)(1,0){2}{\line(0,1){0.2}}

\put(-8,-9.2){\makebox(0,0)[ct]{\(0\)}}
\put(-7,-9.2){\makebox(0,0)[ct]{\(1\)}}


\put(-8,-7){\circle*{0.2}}
\put(-8,-6.8){\makebox(0,0)[cb]{\(\alpha\)}}
\put(-7,-7){\circle{0.2}}
\put(-7,-6.8){\makebox(0,0)[cb]{\(\beta\)}}




\put(-2,-5.8){\makebox(0,0)[cb]{\(U\)}}
\put(-2,-7){\vector(0,1){1}}
\put(-2,-8){\line(0,1){1}}
\multiput(-2.1,-8)(0,1){2}{\line(1,0){0.2}}

\put(-2.2,-7){\makebox(0,0)[rc]{\(1\)}}
\put(-2.2,-8){\makebox(0,0)[rc]{\(0\)}}

\put(1.2,-9){\makebox(0,0)[lc]{\(A\)}}
\put(0,-9){\vector(1,0){1}}
\put(-1,-9){\line(1,0){1}}
\multiput(-1,-9.1)(1,0){2}{\line(0,1){0.2}}

\put(-1,-9.2){\makebox(0,0)[ct]{\(1\)}}
\put(0,-9.2){\makebox(0,0)[ct]{\(2\)}}


\put(-1,-7){\circle{0.2}}
\put(-1,-6.8){\makebox(0,0)[cb]{\(\alpha\)}}
\put(0,-7){\circle{0.2}}
\put(0,-6.8){\makebox(0,0)[cb]{\(\beta\)}}

\put(-1,-8){\circle{0.2}}
\put(-1,-8.2){\makebox(0,0)[ct]{\(\gamma\)}}


\end{picture}
}
\end{figure}
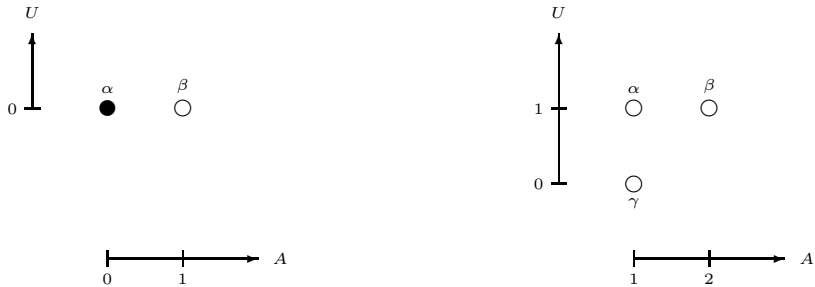


\section{Common Features}
\label{s:Common}
\noindent
In all sections, \S\S\
\ref{s:Septic},
\ref{s:Quintic}, and
\ref{s:Cubic},
the symbol \(\bullet\) indicates a \textit{fine structure} splitting of the
remaining \(\mathbb{F}_p\)-dimension \(R=U+1-A\) into \textit{relative} DPF, and either \textit{capitulation} or \textit{intermediate} DPF.


\section{Conclusion}
\label{s:Conclusion}
\noindent
The purpose of this brief note was to present the fundamental ideas for
a classification of pure \textit{septic} fields \(L=\mathbb{Q}(\sqrt[7]{D})\)
and their Galois closures \(N=\mathbb{Q}(\zeta_7,\sqrt[7]{D})\).
Figure
\ref{fig:Septic}
illustrates the increase of complexity
in comparison with the pure quintic situation in Figure
\ref{fig:MoserExtendedQuintic}:
we shall have \(10\) \textit{coarse} DPF types
instead of only \(6\) (neglecting the quintic splitting of
\(\delta\)/\(\zeta\) and \(\varepsilon\)/\(\eta\)).

Our theory of \textit{fine} DPF types, as developed in
\cite{Ma0,Ma2a}
for \(p=5\),
showed the crucial impact of \textit{splitting} prime divisors
of the conductor \(f\) of \(N/K\)
on the possibility of DPF types
with non-maximal extent of absolute principal factorizations \(A<U+1\),
which will appear in aggravated form for \(p\ge 7\).

On the other hand,
splitting prime divisors of \(f\) have been proved
to enforce non-trivial \(p\)-class numbers of \(L\) and \(N\):
according to Ishida
\cite{Is},
a prime divisor \(\ell\equiv +1\,(\mathrm{mod}\,p)\) of \(f\)
implies \(p\mid h_L\) and \(p\mid h_N\), for any \(p\ge 3\).
Such a prime divisor \(\ell\) splits completely in the cyclotomic field \(K\),
that is, into \(p-1\) prime ideals.
More recently, Kobayashi
\cite{Ky1,Ky2}
has proved that a prime divisor \(\ell\equiv -1\,(\mathrm{mod}\,5)\) of \(f\)
implies \(5\mid h_L\) and \(5\mid h_N\)
and he conjectures the truth of this behavior for \(p\ge 7\).
Such a prime divisor \(\ell\)
splits into \(\frac{p-1}{2}\) prime ideals of \(K\).
Therefore, we were surprised that other splitting prime divisors \(\ell\) of \(f\),
whose occurrence starts with \(p=7\),
do not exert such severe constraints on class numbers,
and we conclude with the following interesting proven phenomenon.

\begin{theorem}
\label{thm:Main}
Let  \(L=\mathbb{Q}(\sqrt[7]{D})\) be a pure septic field
with \(7\)-th power free radicand \(D>1\)
and Galois closure \(N=\mathbb{Q}(\zeta_7,\sqrt[7]{D})\).
If \(D=\ell\equiv 2,4\,(\mathrm{mod}\,7)\) is a prime radicand,
then it \textbf{splits} into \(\frac{7-1}{3}=2\) prime ideals of \(K=\mathbb{Q}(\zeta_7)\).
If the radicand belongs to the range \(2\le D<200\),
then \(\ell\) causes \textbf{relative principal factorizations} in the norm kernel
\((\mathcal{P}_{N/K}/\mathcal{P}_K)\cap\ker(N_{N/L})\),
but \(h_L\) and \(h_N\) are \textbf{not divisible by} \(7\).
\end{theorem}

\begin{proof}
By direct investigation with the aid of the computer algebra system Magma
\cite{MAGMA}.
Explicitly, the radicands are 
\(D\in\lbrace 2,11,23,37,53,67,79,107,109,137,149,151,163,179,191,193\rbrace\). 
\end{proof}


\section{Acknowledgements}
\label{s:Thanks}

\noindent
We gratefully acknowledge that our research was supported by the Austrian Science Fund (FWF):
projects J 0497-PHY and P 26008-N25.



\end{document}